\documentclass[12pt,leqno]{article}
\usepackage{amsfonts}
\usepackage{amsmath, amsthm, amsfonts, amssymb, color}
\usepackage{mathrsfs}
\usepackage[notcite,notref]{showkeys}

\newtheorem{thm}{Theorem}[section]
\newtheorem{cor}[thm]{Corollary}
\newtheorem{lem}[thm]{Lemma}

\theoremstyle{definition}

\newcommand{\scr}[1]{\mathscr #1}
\definecolor{wco}{rgb}{0.5,0.2,0.3}

\numberwithin{equation}{section} \theoremstyle{remark}

\newcommand{\ua}{\uparrow}

\title{{\bf Derivative Formula and Harnack Inequality for Degenerate Functional SDEs}\footnote{Supported in
 part by SRFDP and the Fundamental Research Funds for the Central Universities.}
}
\author{
{\bf  Jianhai Bao$^{b)}$,  Feng-Yu Wang$^{a),b)}$,  Chenggui Yuan$^{b)}$}\\
\footnotesize{$^{a)}$School of Mathematical Sciences,
Beijing Normal
University, Beijing 100875, China}\\
 \footnotesize{$^{b)}$Department of Mathematics,
Swansea University, Singleton Park, SA2 8PP, UK}\\ \footnotesize{wangfy@bnu.edu.cn, F.-Y.Wang@swansea.ac.uk, C.Yuan@swansea.ac.uk}}
\begin{document}
\def\R{\mathbb R}  \def\ff{\frac} \def\ss{\sqrt} \def\B{\mathbf
B}
\def\N{\mathbb N} \def\kk{\kappa} \def\m{{\bf m}}
\def\dd{\delta} \def\DD{\Delta} \def\vv{\varepsilon} \def\rr{\rho}
\def\<{\langle} \def\>{\rangle} \def\GG{\Gamma} \def\gg{\gamma}
  \def\nn{\nabla} \def\pp{\partial} \def\EE{\scr E}
\def\d{\text{\rm{d}}} \def\bb{\beta} \def\aa{\alpha} \def\D{\scr D}
  \def\si{\sigma} \def\ess{\text{\rm{ess}}}
\def\beg{\begin} \def\beq{\begin{equation}}  \def\F{\scr F}
\def\Ric{\text{\rm{Ric}}} \def\Hess{\text{\rm{Hess}}}
\def\e{\text{\rm{e}}} \def\ua{\underline a} \def\OO{\Omega}  \def\oo{\omega}
 \def\tt{\tilde} \def\Ric{\text{\rm{Ric}}}
\def\cut{\text{\rm{cut}}} \def\P{\mathbb P} \def\ifn{I_n(f^{\bigotimes n})}
\def\C{\scr C}      \def\aaa{\mathbf{r}}     \def\r{r}
\def\gap{\text{\rm{gap}}} \def\prr{\pi_{{\bf m},\varrho}}  \def\r{\mathbf r}
\def\Z{\mathbb Z} \def\vrr{\varrho} \def\ll{\lambda}
\def\L{\scr L}\def\Tt{\tt} \def\TT{\tt}\def\II{\mathbb I}
\def\i{{\rm in}}\def\Sect{{\rm Sect}}\def\E{\mathbb E} \def\H{\mathbb H}
\def\M{\scr M}\def\Q{\mathbb Q} \def\texto{\text{o}} \def\LL{\Lambda}
\def\Rank{{\rm Rank}} \def\B{\scr B} \def\i{{\rm i}} \def\HR{\hat{\R}^d}
\def\to{\rightarrow}\def\l{\ell}
\def\8{\infty}

\maketitle

\begin{abstract} By constructing successful couplings, the derivative formula, gradient estimates and  Harnack inequalities are  established for the semigroup associated  with a class of degenerate functional stochastic differential equations.  \end{abstract} \noindent
 \vskip 2cm

\section{Introduction}

In recent years, the coupling argument developed in \cite{ATW} for
establishing dimension-free Harnack inequality in the sense of
\cite{W97} has been intensively applied  to the study of Markov
semigroups associated with a number of stochastic (partial)
differential equations, see e.g. \cite{DRW09, ES, L, LW, O, ORW,
W07, W10, WWX10, WY, WX10, Zh} and references within. In particular,
the Harnack inequalities have been established in \cite{ES, WY} for
a class of non-degenerate functional stochastic differential
equations (SDEs), while the (Bismut-Elworthy-Li type) derivative
formula and applications have been investigated in \cite{GW} for a
class of degenerate SDEs (see also \cite{WZ, Z} for the study by
using Malliavin calculus). The aim of this paper is to establish the
derivative formula and (log-)Harnack inequalities for degenerate
functional SDEs.
 The derivative formula implies explicit gradient estimates of the associated semigroup, while   a number of applications of the (log-)Harnack inequalities
 have been summarized in \cite[\S4.2]{W11} on heat kernel estimates, entropy-cost inequalities, characterizations of invariant measures and contractivity properties of the semigroup.

Let  $m\in \Z_+$ and  $d\in \N$. Denote $\R^{m+d}= \R^m\times\R^d$,  where
$\R^m=\{0\}$ when $m=0$. For $r_0>0$, let
$\C:=C([-r_0,0];\R^{m+d})$ be  the space of continuous functions from
$[-r_0,0]$ into $\mathbb{R}^{m+d}$, which is a Banach space  with the uniform norm
$\|\cdot\|_\infty$. Consider the following functional SDE on
$\R^{m+d}:$
\beq\label{E1} \beg{cases} \d X(t)= \{AX(t)+MY(t)\}\d t,\\
\d Y(t)= \{Z(X(t), Y(t))+b(X_t,Y_t)\}\d t +\si \d B(t),
\end{cases}
\end{equation}
where $B(t)$ is a $d$-dimensional Brownian motion, $\si$ is an
invertible $d\times d$-matrix, $A$ is an $m\times m$-matrix, $M$ is
an $m\times d$-matrix, $Z:  \R^m\times\R^d \to \R^d$ and $ b:  \C
\to \R^d$ are locally Lipschitz continuous (i.e. Lipschitzian on
compact sets), $(X_t, Y_t)_{t\ge 0}$ is  a process on $\C$
 with $(X_t,Y_t) (\theta):= (X(t+\theta), Y(t+\theta)), \theta\in [-r_0,0].$
We assume that there exists an integer number $0\le k\le m-1$ such that
\beq\label{RR} \text{Rank}[M, AM, \cdots, A^kM]=m.\end{equation} When $m=0$ this condition automatically holds by convention.
Note that when $m\ge 1$,  this rank condition holds for some $k> m-1$ if and only if it holds for $k=m-1.$

Let $\nn,\nn^{(1)}$ and $\nn^{(2)}$ denote the gradient operators on $\R^{m+d}, \R^m$ and $\R^d$ respectively, and let
\beg{equation*}\beg{split} L f(x,y):= &\<Ax+My, \nn^{(1)}f(x,y)\> +\<Z(x,y), \nn^{(2)}f(x,y)\>\\
&+\ff 1 2 \sum_{i,j=1}^d (\si\si^*)_{ij} \ff{\pp^2}{\pp y_i\pp y_j}
f(x,y), \ \ (x,y)\in \R^{m+d}, f\in
C^2(\R^{m+d}).\end{split}\end{equation*} Since both $Z$ and $b$ are
locally Lipschitz continuous, due to \cite{SR}
 the equation (\ref{E1}) has a unique local solution for any initial data $(X_0,Y_0)\in\C$.
 To ensure the    non-explosion  and further regular properties of the solution, we make use of the following assumptions:

\paragraph{(A)} \ \emph{  There exist    constants $\ll,l> 0$  and $W\in C^2(\R^{m+d})$ of compact level sets  with $W\ge 1$ such that
\beg{enumerate}
\item[$(A1)$] $ LW\le \ll W,\ \ |\nn^{(2)}W|\le \ll W;$
\item[$(A2)$] $ \<b(\xi), \nn^{(2)}W(\xi(0))\> \le \ll  \|W(\xi)\|_\8,\ \ \xi \in \C;$
\item[$(A3)$] $| Z(z)-Z(z')| \le   \ll |z-z'| W(z')^l,\ \ z,z'\in\R^{m+d}, |z-z'|\le 1;$
\item[$(A4)$] $|b(\xi)-b(\xi')|  \le \ll \|\xi-\xi'\|_\8 \|W(\xi')\|_\8^l, \, \xi,\xi'  \in \C, \|\xi-\xi'\|_\infty\le 1.$
    \end{enumerate} }

\

Comparing with the framework investigated in \cite{GW,Z}, where
$b=0, A=0$ and Rank$[M]=m$ are assumed, the present model is
more general and the segment process we are going to investigate is
an infinite-dimensional Markov process. On the other hand,  unlike
in \cite{GW} where the condition $|\nn^{(2)}W|\le \ll W$ is not
used, in the present setting this condition seems essential in order
to derive moment estimates of the segment process (see the proof of
Lemma \ref{lem1} below). Moreover, if $|\nn W|\le c W$ holds for
some constant $c>0$, then $(A3)$ and $(A4)$ hold for some $\ll>0$ if
and only if there exists a constant $\ll'>0$ such that $|\nn Z|\le
\ll' W^l$  and $|\nn b|\le \ll' \|W\|_\infty^l$ holds on $\R^{m+d}$
and $\C$ respectively.

It is easy to see that {\bf (A)} holds for $W(z)= 1+|z|^2$, $l=1$
and some constant $\ll>0$ provided that $Z$ and $b$ are globally
Lipschitz continuous on $\R^{m+d}$
 and $\C$ respectively. It is clear that $(A1)$ and $(A2)$ imply the non-explosion of the solution (see Lemma \ref{lem1} below). In this paper
we aim to investigate regularity properties of the Markov semigroup associated with the segment process:
$$P_t f(\xi)= \E^\xi f(X_t,Y_t),\ \ f\in \B_b(\C), \xi\in\C,$$ where $\B_b(\C)$ is the class of all bounded measurable functions on $\C$ and
$\E^\xi$ stands for the expectation for the solution starting at the
point $\xi\in\C$. When $m=0$ we have $X_t\equiv 0$ and
$\C=\{0\}\times\C_2\equiv\C_2:=C([-r_0,0];\R^d)$, so that $P_tf$ can
be simply formulated as $P_t f(\xi)=\E^\xi f(Y_t)$ for $f\in
\B_b(\C_2), \xi\in \C_2.$  Thus, (\ref{E1}) also includes
non-degenerate functional SDEs. For any $h=(h_1, h_2)\in\C$ and
$z\in \R^{m+d}$, let $\nn_h$ and $\nn_z$ be the directional
derivatives along $h$ and $z$ respectively. The following result
provides an explicit derivative formula for $P_T, T>r_0.$

\begin{thm}\label{T1.1} Assume
  {\bf (A)} and let $T>r_0$.  Let $v:[0,T]\to \R$ and $\aa: [0,T]\to \R^m$ be Lipschitz continuous such that
$v(0)=1, \aa(0)=0, v(s)=0,  \aa(s)=0 $ for $s\ge T-r_0,$ and
\beq\label{LL} h_1(0)+\int_0^t\e^{-sA}M\phi(s)\d s =0,\ \ \ t\ge
T-r_0,\end{equation} where  $\phi(s):= v(s) h_2(0)+\aa(s)$.  Then
for any $h=(h_1, h_2)\in \C$ and $f \in \B_b(\C),$ \beq\label{Bis}
\nn_hP_T f(\xi)
 =\E^\xi \bigg\{f(X_T, Y_T)\int_0^{T}\big\<N(s), (\si^*)^{-1}\d B(s)\big\>\bigg\},\ \ \ \xi\in\C\end{equation} holds for
 $$N(s):=  (\nn_{\Theta(s)}Z)(X(s), Y(s))
 + (\nn_{\Theta_s}b)(X_s, Y_s)-v'(s)h_2(0)-\aa'(s),\ \ \ s\in [0,T],$$ where
 $$\Theta(s)=(\Theta^{(1)}(s), \Theta^{(2)}(s)):= \beg{cases} h(s),\ \  &\text{if}\  s\le 0,\\
\big(\e^{As}h_1(0) +\int_0^{s} \e^{(s-r)A} M\phi(r)\d r,\ \phi(s)\big), &\text{if}\
s>0.\end{cases} $$
 \end{thm}

A simple choice of $v$ is
$$v(s)= \ff{(T-r_0-s)^+}{T-r_0},\ \ \ s\ge 0.$$ To present a specific choice of $\aa$, let
$$Q_t:= \int_0^t \ff{s(T-r_0-s)^+}{(T-r_0)^2}\e^{-sA} MM^*\e^{-s A^*}\d s,\ \ \ t>0.$$ According to \cite{S} (see also \cite[Proof of Theorem 4.2(1)]{WZ}),  when $m\ge 1$ the matrix $Q_t$ is invertible with
\beq\label{QQ} \|Q_t^{-1}\|\le c(T-r_0)(t\land 1)^{-2(k+1)},\ \ t>0\end{equation} for some constant $c>0.$

\begin{cor}\label{C1.2} Assume
  {\bf (A)} and let $T>r_0$.  Then $(\ref{Bis})$ holds for $v(s)=\ff{(T-r_0-s)^+}{T-r_0}$ and
$$\aa(s)= -\ff{s(T-r_0-s)^+}{(T-r_0)^2}M^*\e^{-sA^*}Q_{T-r_0}^{-1} \bigg(h_1(0)+\int_0^{T-r_0} \ff{(T-r_0-r)^+}{T-r_0} \e^{-rA}Mh_2(0)\d r\bigg),$$ where by convention 
$M=0\ ($hence, $\aa=0)$ if $m=0$. \end{cor}

The following gradient estimates are direct consequences of Theorem \ref{T1.1}.

\beg{cor}\label{C1.3} Assume  {\bf (A)}.   Then:
\beg{enumerate}\item[$(1)$] There exists a constant $C\in
(0,\infty)$ such that
\beg{equation*} \beg{split}|\nn_h P_Tf(\xi)|\le C &\ss{P_T f^2(\xi)} \bigg\{ |h(0)|\Big(1 + \ff{\|M\|}{(T-r_0)^{2k+1}\land 1}\Big)\\
&+  \|W(\xi)\|_\infty^l\ss{T\land (1+r_0)} \Big(\|h\|_\infty+\ff{\|M\|\cdot |h(0)|}{(T-r_0)^{2k+1}\land 1}\Big)\bigg\}
\end{split}\end{equation*} holds for all $T>r_0, \xi,h\in\C$ and $f\in \B_b(\C)$;
\item[$(2)$] Let $|\nn^{(2)}W|^2\le \dd W$ hold for some constant $\dd>0.$ If $l\in [0,1/2)$ then there exists a constant $C\in (0,\infty)$ such that
\beg{equation*}\beg{split} &|\nn_h P_T f(\xi)|\le r \big\{P_Tf\log
f- (P_Tf)\log P_T f\big\}(\xi)\\
& \qquad+  \ff{CP_Tf(\xi)}r\bigg\{|h(0)|^2\bigg(
 \ff {1} {(T-r_0)\land 1} +\ff {\|M\|^2} {\{(T-r_0)\land 1\}^{4k+3}}\bigg)\\
 &\qquad+ \|h\|_\infty^2\|W(\xi)\|_\infty +\bigg(\|h\|_\infty^2 + \ff{|h(0)|^2\|M\|^2}{\{(T-r_0)\land 1\}^{4k+2}}\bigg)^{\ff 1 {1-2l}} \bigg(\ff{r^2}{\|h\|_\infty^2}\lor 1\bigg)^{\ff{2l}{1-2l}}\bigg\}\end{split}\end{equation*} holds for  all $r>0, T>r_0,
\xi,h\in\C$ and positive $f\in \B_b(\C)$;
\item[$(3)$] Let $|\nn^{(2)}W|^2\le \dd W$ hold for some constant $\dd>0.$ If $l=\ff 1 2$ then there exist constants $C,C'\in (0,\infty)$ such that
\beg{equation*}\beg{split}|\nn_h P_T f(\xi)|\le & r \big\{P_Tf\log f- (P_Tf)\log P_T f\big\}(\xi)\\
&  + \ff{CP_Tf(\xi)}r\bigg\{  |h(0)|^2\bigg(
 \ff {1} {(T-r_0)\land 1} +\ff {\|M\|^2} {\{(T-r_0)\land 1\}^{4k+3}}\bigg)\\
 & +\|W(\xi)\|_\infty\Big(\|h\|_\infty^2+\ff{\|M\|^2|h(0)|^2}{\{(T-r_0)\wedge1\}^{4k+2}}\Big)\bigg\}\end{split}\end{equation*}
  holds for  $$r\ge C'\bigg(\|h\|_\infty+\ff{\|M\|\cdot|h(0)|}{\{(T-r_0)\wedge1\}^{2k+1}}\bigg),$$ all $T>r_0, \xi,h\in\C$ and positive $f\in \B_b(\C)$.
\end{enumerate} When $m=0$ the above assertions hold with $\|M\|=0$. \end{cor}

According to \cite{ATW09}, the entropy gradient estimate implies the Harnack inequality with power, we have the following result which follows immediately from Corollary \ref{C1.3} (2)   and \cite[Proposition 4.1]{GW}. Similarly, Corollary \ref{C1.3} (3) implies the same type Harnack inequality for smaller $\|h\|_\infty$ comparing to $T-r_0$.

 \beg{cor}\label{C1.4} Assume  {\bf (A)} and let $|\nn^{(2)}W|^2\le \dd W$ hold for some constant $\dd>0.$  If $l\in [0,\ff 1 2)$ then there exists a constant $C\in (0,\infty)$ such that
 \beg{equation*}\beg{split}  (P_Tf)^p(\xi+h)\le  &P_Tf^p(\xi)
  \exp\bigg[\ff{Cp}{p-1} \bigg\{\|h\|_\infty^2 \int_0^1\|W(\xi+sh)\|_\infty\d s\\
   & \quad + \bigg(\|h\|_\infty^2+\ff{\|M\|^2|h(0)|^2}{\{(T-r_0)\land 1\}^{4k+2}}\bigg)^{\ff 1 {1-2l}}\bigg(\ff{(p-1)^2}{\|h\|_\infty^2}\lor 1\bigg)^{\ff{2l}{1-2l}} \bigg\}\bigg]\end{split}\end{equation*}
 holds for all $T>r_0, p>1, \xi,h\in \C$ and positive $f\in \B_b(\C).$  If $m=0$ then the assertion holds for $\|M\|=0.$
 \end{cor}

Finally, we consider  the log-Harnack inequality introduced in \cite{RW, W10b}. To this end, as in \cite{GW}, we slightly strengthen $(A3)$ and $(A4)$ as for follows:
there exists an increasing function $U$ on $[0,\infty)$ such that \beg{enumerate}\item[$(A3')$] $| Z(z)-Z(z')| \le   \ll |z-z'| \big\{W(z')^l+U(|z-z'|)\big\},\ \ z,z'\in\R^{m+d};$
\item[$(A4')$] $|b(\xi)-b(\xi')|  \le \ll \|\xi-\xi'\|_\8 \big\{\|W(\xi')\|_\8^l+U(\|\xi-\xi'\|_\8)\big\}, \, \xi,\xi'  \in \C.$
    \end{enumerate}  Obviously, if $$W(z)^l\le c\{W(z')^l +U(|z-z'|)\},\ \ z,z'\in \R^{m+d}$$ holds for some constant $c>0$, then $(A3)$ and $(A4)$ imply $(A3')$ and $(A4')$ respectively with possibly different $\ll.$

\beg{thm}\label{T1.5} Assume $(A1), (A2), (A3')$ and $(A4')$.  Then
there exists a constant $C\in (0,\infty)$ such that for any positive
$f\in \B_b(\C), T>r_0$ and $\xi,h\in\C$, \beg{equation*}\beg{split}
&P_T\log f(\xi+h)- \log P_T f(\xi)\le C
\bigg\{\bigg[\|W(\xi+h)\|_\8^{2l} +U^2\Big(C\|h\|_\8
+\ff{C\|M\|\cdot|h(0)|}{(T-r_0)\land 1}\Big)\bigg]\|h\|_\infty^2 \\
&\qquad\qquad\ \ \ \ \ \ \ \ \ \ \ \ \ \ \ \ \ \ \ \ \ \ \ \ \ \ \ \
\ \  +\ff{|h(0)|^2}{(T-r_0)\wedge1}
+\ff{\|M\|^2|h(0)|^2}{\{(T-r_0)\land
1\}^{4k+3}}\bigg\}.\end{split}\end{equation*}If $m=0$ then the
assertion holds for $\|M\|=0.$
\end{thm}

For applications of the Harnack and log-Harnack inequalities we are referred to \cite[\S 4.2]{W11}. The remainder of the paper is organized as follows: Theorem \ref{T1.1} and Corollary \ref{C1.2} are proved Section 2, while Corollary \ref{C1.3} and Theorem \ref{T1.5} are proved in Section 3; in Section 4  the  assumption {\bf(A)} is weakened for the discrete time delay case, and two examples are presented to illustrate our results.

\section{Proofs of Theorem \ref{T1.1} and Corollary \ref{C1.2}}

 \beg{lem}\label{lem1}
      Assume $(A1)$ and $(A2)$. Then for any $k>0$ there exists a constant $C>0$ such that
$$\E^\xi \sup_{-r_0\le s \le t}W(X(s),Y(s))^k \le 3\|W(\xi)\|_\infty^k \e^{Ct},\ \ t\ge 0,\ \xi\in\C$$ holds. Consequently, the solution is non-explosive.
      \end{lem}

\beg{proof} For any $n\ge 1$, let  $$ \tau_n:= \inf\{t\in [0, T]:
|X(t)|+|Y(t)|\ge n\}.
$$ Moreover, let
$$\l(s):= W(X,Y)(s),\ \ \ s\ge -r_0.$$
By the It\^o formula and using the first inequality in $(A1)$ and $(A2)$ we may find a constant $C_1>0$ such that

\begin{equation}\label{W2}\beg{split}& \l(t\land\tau_n)^k = \l(0)^k +k \int_0^{t\land\tau_n} \l(s)^{k-1} \<\nn^{(2)}W(X,Y)(s), \si\d B(s)\>\\
 &\quad +k \int_0^{t\land\tau_n}\l(s)^{k-1} \Big\{LW(X,Y)(s)+\big\<b(X_s,
 Y_s), \nn^{(2)}W(X,Y)(s)\big\>\\
 &\quad+\frac{1}{2}(k-1)\l(s)^{-1}|\sigma^*\nn^{(2)}W(X,Y)(s)|^2\Big\}\d s\\
&\le l(0)^k+ k \int_0^{t\land\tau_n} \l(s)^{k-1}
\<\nn^{(2)}W(X,Y)(s), \si\d B(s)\> +C_1 \int_0^{t\land\tau_n}
\sup_{r\in [-r_0, s]}\l(r)^k\d s.
\end{split}\end{equation} Noting that    by  the second inequality
in $(A1)$ and the Burkholder-Davis-Gundy inequality we obtain

\beg{equation*}\beg{split} &k\E^\xi \sup_{s\in [0,t]}
\Big|\int_0^{s\land \tau_n}  \l(r)^{k-1} \<\nn^{(2)}W(X,Y)(s), \si\d
B(r)\>\Big|\le C_2
\E^\xi\bigg( \int_0^{t} \l(s\land\tau_n)^{2k}\d s\bigg)^{1/2}\\
&\le C_2 \E^\xi\bigg\{ \Big(\sup_{s\in [0,t]} \l(s\land \tau_n)^k\Big)^{1/2}\bigg(\int_0^{t} \l(s\land\tau_n)^k\d s\bigg)^{1/2}\bigg\}\\
&\le \ff 1 2 \E^\xi \sup_{s\in [0,t]} \l(s\land \tau_n)^k + \ff
{C_2^2} 2 \E^\xi\int_0^t \sup_{r\in [0,s]} \l(r\land\tau_n)^k\d
s\end{split}
\end{equation*} for some constant $C_2>0$. Combining this with
(\ref{W2}) and noting that $(X_0,Y_0)=\xi$, we conclude that there
exists a constant $C>0$ such that
 $$\E^\xi \sup_{-r_0\le s\le t}\l(s\land\tau_n)^k  \le 3\|W(\xi)\|_\infty^k
+ C\E^\xi \int_0^{t} \sup_{s\in [-r_0,t]} \l(s)^k\d s,\ \ \ t\ge 0.
$$
Due to the  Gronwall lemma this implies that
$$\E^\xi \sup_{-r_0\le s\le t}\l(s\land\tau_n)^k \le 3\|W(\xi)\|_\infty^k\e^{Ct},\ \ \ t\ge 0, n\ge 1.$$
Consequently, we have $\tau_n\uparrow \infty$ as $n\uparrow\infty$,
and thus the desired inequality follows by letting $n\to\infty$.
\end{proof}

 To establish the derivative formula, we first construct couplings for solutions starting from $\xi$ and $\xi+\vv h$ for $\vv\in (0,1],$ then let $\vv\to 0$.
For fixed $\xi=(\xi_1, \xi_2), h=(h_1,h_2)\in \scr C$, let $(X(t), Y(t))$ solve (\ref{E1})  with $(X_0,Y_0)= \xi$; and for any $\vv\in (0, 1]$, let $(X^\vv(t), Y^\vv(t))$ solve the equation
\beq\label{E2} \beg{cases} \d X^\vv(t)= \{AX^\vv(t)+MY^\vv(t)\}\d t,\\
\d Y^\vv(t)= \{Z(X(t), Y(t))+b(X_t,Y_t)\}\d t +\si \d B(t)+ \vv\{v'(t)h_2(0)+\aa'(t)\}\d t\end{cases}\end{equation} with
$(X_0^\vv,Y_0^\vv)= \xi+\vv h.$
By Lemma \ref{lem1} and (\ref{EE}) below,
the solution to (\ref{E2}) is non-explosive as well.

\beg{prp} \label{Pro1} Let $\phi(s):= v(s)h_2(0)+\aa(s),\  s\in [0,
T]$, and  the conditions of Theorem \ref{T1.1} hold. Then
\beq\label{EE} (X^\vv(t),Y^\vv(t))=(X(t),Y(t))+\vv\Theta(t), \ \
\vv, t\ge 0\end{equation} holds for
$$\Theta(t):=(\Theta^{(1)}(t), \Theta^{(2)}(t)):= \beg{cases} h(t),\ \  &\text{if}\  t\le 0,\\
\big(\e^{At}h_1(0) +\int_0^{t} \e^{(t-r)A} M\phi(r)\d r,\
\phi(t)\big), &\text{if}\ t>0.\end{cases} $$ In particular,
$(X_T^\vv, Y_T^\vv)=(X_T,Y_T).$\end{prp}

\beg{proof} By (\ref{E2}) and noting that $v(0)=1$ and $v(s)=0$ for $s\ge T-r_0$,  we have
$Y^\vv(t)= Y(t)+ \vv\phi(t)$ and
$$X^\vv(t)= X(t) +\vv \e^{At} h_1(0) +\vv \int_0^t \e^{(t-s)A}M\phi(s)\d s,\ \ t\ge 0.$$  Thus, (\ref{EE}) holds. Moreover, since $\aa(s)=v(s)=0$ for $s\ge T-r_0$, we
have $\Theta^{(2)}(s)=\phi(s)=0$  for $s\ge T-r_0$. Moreover, by (\ref{LL}) we have $\Theta^{(1)}(s)=0$ for $s\ge T-r_0.$ Therefore, the proof is finished.
\end{proof}

Since according to Proposition \ref{Pro1} we have  $(X_T^\vv,Y_T^\vv)=(X_T,Y_T)$. Noting that $(X_0^\vv,Y_0^\vv)=\xi+\vv h$,
if  (\ref{E2}) can be formulated as (\ref{E1}) using a different Brownian motion, then we are able to link $P_Tf(\xi)$  to $P_Tf(\xi+\vv h)$ and furthermore
derive the derivative formula by taking derivative w.r.t. $\vv$ at $\vv=0.$ To this end, let
$$
\Phi^\vv(s)=Z(X(s), Y(s))-Z(X^\vv(s),
Y^\vv(s))+b(X_s,Y_s)-b(X^\vv_s,Y^\vv_s)+\vv\{v'(s)h_2(0)+\aa'(s)\}.
$$
Set
\begin{align*}
R^\vv(s)=\exp\bigg[-\int_0^s\< \si^{-1}\Phi^\vv(r), \d
B(r)\>-\frac{1}{2}\int_0^s |\si^{-1}\Phi^\vv(r)|^2\d r\bigg],
\end{align*}
and
$$
B^\vv(s)=B(s)+\int_0^s\si^{-1}\Phi^\vv(r)\d r.
$$ Then (\ref{E2}) reduces to
\beq\label{E2'}   \beg{cases} \d X^\vv(t)= \{AX^\vv(t)+MY^\vv(t)\}\d t,\\
\d Y^\vv(t)= \{Z(X^\vv(t), Y^\vv(t))+b(X_t^\vv,Y_t^\vv)\}\d t +\si \d B^\vv(t).
\end{cases}
\end{equation}  According to the Girsanov theorem, to ensure that $B^\vv(t)$ is a Browanian motion under $\Q_\vv:= R^\vv(T)\P$, we first prove that
$R^\vv(t)$ is an exponential martingale. Moreover, to obtain the derivative formula using the dominated convergence theorem, we also need $\{\ff{R^\vv(T)-1}\vv\}_{\vv\in (0,1)}$ to be uniformly integrable.  Therefore, we will need the following two lemmas.

\begin{lem}\label{L2.2}
Let  {\bf (A)} hold. Then there exists $\vv_0>0$ such that
$$
\sup_{s \in [0, T], \vv \in (0, \vv_0)}\E [R^\vv(s)\log R^\vv(s) ]<\8,
$$
so that for each $\vv \in (0, 1), $ $(R^\vv(s))_{s \in [0, T]}$ is a
uniformly integrable martingale.
\end{lem}

\begin{proof} By (\ref{EE}), there exists  $\vv_0>0$ such that
\beq\label{ED} \vv_0|\Theta(t)|\le 1, \ \ \ t\in [-r_0,T].\end{equation} For any $\vv\in [0, \vv_0],$ define
   $$
\tau_n:=\inf\{t\ge 0: |X(t)|+|Y(t)|+|X^{\vv }(t)|+|Y^{\vv }(t)|\ge
n\}, \, n\ge 1.
$$ We have $\tau_n\uparrow\infty$ as $n\uparrow\8$ due to the non-explosion.
By the Girsanov theorem, the process $\{R^\vv(s\wedge\tau_n)\}_{s\in [0,T]}$ is a
martingale and $\{B^\vv(s)\}_{s\in [0,T\land \tau_n]}$ is a Brownian motion under the
probability measure $\Q_{\vv,n}:=R^\vv(T\wedge\tau_n)\P.$  By the
definition of $R^\vv(s)$ we have \beq \label{2.6}\E [R^\vv(s\wedge\tau_n)\log
R^\vv(s\wedge\tau_n)]=\E_{\Q_{\vv,n}}[\log R^\vv(s\wedge\tau_n)]\le
\frac{1}2\E_{\Q_{\vv,n}}\int_0^{T\wedge
\tau_n}|\si^{-1}\Phi^\vv(r)|^2\d r.
\end{equation}
By (\ref{ED}), $(A3)$ and $(A4)$,
\beq\label{2.7}
|\sigma^{-1}\Phi^{\vv}(s)|^2\le c
\vv^2\|W(X_s^\vv, Y_s^\vv)\|_\8^{2l},
\end{equation}
holds for some constant $c$ independent of $\vv$.
By the weak uniqueness of the solution to (\ref{E1}) and (\ref{E2'}), the distribution of   $(X^\vv(s), Y^\vv(s))_{s\in[0,T\land\tau_n]}$   under ${ \Q_{\vv,n}}$ coincides with that of the solution to (\ref{E1}) with $(X_0,Y_0)=\xi+\vv h$ up to time $T\land\tau_n$,  we therefore obtain from Lemma \ref{lem1} that
$$
\E [R^\vv(s\wedge\tau_n)\log R^\vv(s\wedge\tau_n)]\le  c\|W(\xi+\vv
h)\|_\8^{2l}\int_0^T \e^{Ct}\d t<\infty,\ \ \ n\ge 1, \vv\in
(0,\vv_0).$$ Then the  required assertion follows by letting $n \to
\8.$
\end{proof}

\begin{lem}\label{L2.3}
If {\bf (A)}  holds,  then there exists $\vv_0>0$ such that
\begin{align*}
\sup_{\vv\in (0, \vv_0)}\E \left(\ff{R^\vv(T)-1}{\vv}\log \ff{R^\vv(T)-1}{\vv}\right)< \infty.
\end{align*}
Moreover,
\beq\label{y2}\beg{split}
 &\lim_{\vv\to 0 } \ff{R^\vv(T)-1}{\vv}=\\
&\int_0^{T}\big\< (\nn_{\Theta(s)}Z)(X(s), Y(s))+(\nn_{\Theta_s}b)(X_s, Y_s)-v'(s)h_2(0)-\aa'(s), (\sigma^*)^{-1}\d B(s)\big\>.
\end{split}\end{equation}
\end{lem}

\begin{proof} Let $\vv_0$ be such that (\ref{ED}) holds. Since (\ref{y2}) is a direct consequence of (\ref{EE}) and the definition of $R^\vv(T)$,
  we   only prove the first assertion.
By \cite{GW} we know that
$$
\ff{R^\vv(T)-1}{\vv}\log \ff{R^\vv(T)-1}{\vv}\le 2 R^\vv(T)\bigg(\ff{\log R^\vv(T)}{\vv}\bigg)^2.
$$
Since due to Lemma \ref{L2.2} $\{B^\vv(t)\}_{t\in [0, T]}$ is a Brownian motion under the probability measure $\Q_\vv:=R^\vv(T)\P$, and since
\begin{align*}
\log R^\vv(T)&=-\int_0^T\< \si^{-1}\Phi^\vv(r), \d
B(r)\>-\frac{1}{2}\int_0^T |\si^{-1}\Phi^\vv(r)|^2\d r\\
&=-\int_0^T\< \si^{-1}\Phi^\vv(r), \d
B^\vv(r)\>+\frac{1}{2}\int_0^T |\si^{-1}\Phi^\vv(r)|^2\d r,
\end{align*}
it follows from (\ref{2.7})  that

\begin{equation*}
\begin{split}
&\E \bigg(\ff{R^\vv(T)-1}{\vv}\log \ff{R^\vv(T)-1}{\vv}\bigg)\le \E \bigg(2 R^\vv(T)\bigg(\ff{\log R^\vv(T)}{\vv}\bigg)^2\bigg)=2 \E_{\Q_\vv}\bigg(\ff{\log R^\vv(T)}{\vv}\bigg)^2\\
&\le \ff{4}{\vv^2}\E_{\Q_\vv}\bigg(\int_0^T\< \si^{-1}\Phi^\vv(r), \d B^\vv(r)\>\bigg)^2+\ff{1}{\vv^2}\E_{\Q_\vv}\bigg( \int_0^T |\si^{-1}\Phi^\vv(r)|^2\d r\bigg)^2\\
&\le  \ff{4}{\vv^2}\int_0^T\E_{\Q_\vv}| \si^{-1}\Phi^\vv(r)|^2\d r+\ff{T}{\vv^2}  \int_0^T \E_{\Q_\vv}|\si^{-1}\Phi^\vv(r)|^4\d r \\
&\le c\int_0^T\E_{\Q_\vv}\|W(X_r^\vv, Y_r^\vv)\|_\8^{4l}\d r
\end{split}
\end{equation*} holds for some constant $c>0$.
As explained in the proof of Lemma \ref{L2.2}  the distribution of   $(X_s^\vv, Y_s^\vv)_{s\in [0,T]} $ under $\Q_\vv$ coincides with that of the segment process of the solution to
(\ref{E1}) with $(X_0,Y_0)=\xi+\vv h$,    the first assertion follows by Lemma \ref{lem1}.
\end{proof}

\begin{proof}[Proof of Theorem \ref{T1.1}]
Since Lemma \ref{L2.2}, together with  the Girsanov theorem,  implies that $\{B^\vv(s)\}_{s\in [0,T]}$ is a Brownian motion
with respect to ${\mathbb{Q}}_\vv:= R^\vv(T) {\mathbb{P}},$ by (\ref{E2'}) and $(X_T,Y_T)=(X_T^\vv,Y_T^\vv)$ we obtain
\beq\label{y1}
P_Tf(\xi+\vv h)=\E_{\Q_\vv}f(X_T^\vv, Y_T^\vv)=\E\{R^\vv(T)f(X_T, Y_T)\}.
\end{equation}
Thus,
$$P_Tf(\xi+\vv h)-P_Tf(\xi) =\E{R^\vv(T)f(X_T, Y_T)}-\E{f(X_T, Y_T)}
 =\E[(R^\vv(T)-1)f(X_T, Y_T)].$$
Combining this with  Lemma \ref{L2.3} and using the dominated convergence theorem, we arrive at
\begin{equation*}
\begin{split}
\nn_hP_T f(\xi, \eta)&=\lim_{\vv\to 0}\ff{P_Tf(\xi+\vv h)-P_Tf(\xi)}{\vv}
=\lim_{\vv\to 0}\ff{\E[(R^\vv(T)-1)f(X_T, Y_T)]}{\vv}\\
&=\E \bigg\{f(X_T, Y_T)\int_0^{T}\big\<N(s), (\sigma^*)^{-1}\d B(s)\big\> \bigg\}.\end{split}
\end{equation*}
\end{proof}

\beg{proof}[Proof of Corollary \ref{C1.2}] It suffices to verify (\ref{LL}) for the specific $v$ and $\aa$.  Since when $m=0$ we have $h_1=M=0$ so that (\ref{LL}) trivially holds, we only consider $m\ge 1.$
In this case, (\ref{LL}) is satisfied
 since according to  the definition of $\phi(s)$ and $\aa(s)$ we have for $t\ge T-r_0$,
\beg{equation*}\beg{split} &\int_0^t \e^{-sA}M\phi(s) \d s = \int_0^{T-r_0} \e^{-sA}M\phi(s)\d s\\
& = \int_0^{T-r_0}v(s)\e^{-sA}Mh_2(0)\d s  -Q_{T-r_0}Q_{T-r_0}^{-1}
\bigg(h_1(0) +\int_0^{T-r_0} v(s) \e^{-s A} M h_2(0)\d s\bigg)\\
& =-h_1(0).\end{split}\end{equation*}\end{proof}

\section{Proofs of Corollary \ref{C1.3} and Theorem \ref{T1.5}}

To prove the entropy-gradient estimates in Corollary (2) and (3), we need the following simple lemma which seems new and might be interesting by itself.

\beg{lem} \label{L3.1} Let $\l(t)$ be a non-negative  continuous
semi-martingale and let $\M(t)$ be a continuous martingale with
$\M(0)=0$ such that
$$\d \l(t)\le \d\M(t)+c \bar \l_t\d t, $$ where $c\ge 0$ is a constant and $\bar \l_t:=\sup_{s\in [0,t]}\l(s).$ Then
$$\E\exp\bigg[\ff{\vv}{T\e^{1+cT}}\int_0^T\bar \l_t\d t\bigg]\le \e^{\vv \l(0) + 1} \big(\E\e^{2\vv^2\<\M\>(T)}\big)^{1/2},\ \ T,\vv\ge 0.$$  \end{lem}

\beg{proof} Let $\bar \M_t:= \sup_{s\in [0,t]}\M(t).$ We have
$$\bar \M_t+ c\int_0^t\bar \l_s\d s \ge \bar \l_t-\l(0).$$ Thus,
\beg{equation*}\beg{split} &\ff{ \l_T}{\e^{1+cT}}-\l(0)\le\ff{\bar\M_T+c\int_0^T\bar \l_t\d t}{\e^{1+cT}} -\big(1-\e^{-(1+cT)}\big)\l(0)\\
& = \int_0^T\d \bigg\{\e^{-(c+T^{-1})t} \bigg(\bar\M_t + c \int_0^t \bar \l_s\d s\bigg)\bigg\}-\big(1-\e^{-(1+cT)}\big)\l(0)\\
&= \int_0^T \e^{-(T^{-1}+c)t} \d \bar\M_t + \int_0^T\e^{-(c+T^{-1})t}\bigg\{c\bar \l_t  -   (T^{-1}+c)  \bigg(\bar\M_t+c\int_0^t\bar \l_s\d s\bigg)\bigg\} \d t\\
&\qquad\qquad\qquad -\big(1-\e^{-(1+cT)}\big)\l(0)\\
&\le \bar \M_T+ \int_0^T\e^{-(c+T^{-1})t}\Big\{c\bar \l_t  -   (T^{-1}+c)  \big(\bar \l_t-\l(0)\big)\Big\} \d t-\big(1-\e^{-(1+cT)}\big)\l(0)\\
&\le\bar \M_T -\ff 1 {T\e^{1+cT}} \int_0^T \bar \l_t\d
t.\end{split}\end{equation*} Combining this with
$$\E\e^{\vv \bar \M_t}\le  \E\e^{1+ \vv \M(T)}\le \e\big(\E\e^{2\vv^2\<\M\>(T)}\big)^{1/2},$$ we complete the proof.
\end{proof}

\beg{cor}\label{C3.1}
Assume  {\bf (A)} and let  $|\nn^{(2)}W|^2\le \dd W$ hold for some constant $\dd>0.$  Then there exists a constant $c>0$ such that
\beg{equation*}\beg{split} &\E^\xi \exp\bigg[ \ff{1}{2\|\si\|^2\dd T^2\e^{2+2cT}} \int_0^T  \|W(X_t,Y_t)\|_\infty \d t \bigg]\\
&\le \exp\bigg[2 +\ff{W(\xi(0))}{\|\si\|^2\dd T
\e^{1+cT}}+\ff{r_0\|W(\xi)\|_\8}{2\|\si\|^2\dd T^2\e^{2+2cT}}\bigg],
\ \ T> r_0.\end{split}\end{equation*}\end{cor}

\beg{proof} By {\bf (A)} and the It\^o formula, there exists a constant $c>0$ such that
$$
\d W(X,Y)(s)
\le  \<\nn^{(2)}W(X,Y)(s), \si \d B(s)\>+c\|W(X_s,Y_s)\|_\8\d s.$$
  Let
$$\M(t):=  \int_0^t \<\nn^{(2)}W(X,Y)(s), \si \d B(s)\>,\ \ l(t):= W(X,Y)(t),$$ and let $\vv= (2\|\si\|^2\dd T\e^{1+cT})^{-1}$ such that
$$  \ff{\vv} {T\e^{1+cT}} = 2\|\si\|^2\vv^2.$$ Then by Lemma \ref{L3.1} and $|\nn^{(2)}W|^2\le\dd W$, we have
\beg{equation*}\beg{split} &\E^\xi \exp\bigg[\ff \vv{T\e^{1+cT}} \int_0^T\bar l_t\d t \bigg] \le \e^{\vv l(0) + 1} \big(\E^\xi\e^{2\vv^2\<\M\>(T)}\big)^{1/2}\\
&\le \e^{1+ \vv l(0)}\Big(\E^\xi \e^{2\vv^2 \|\si\|^2\dd \int_0^T
\bar l_t\d t}\Big)^{1/2} =
  \e^{1+ \vv l(0)}\Big(\E^\xi \e^{\ff \vv{T\e^{1+cT}} \int_0^T \bar l_t\d t}\Big)^{1/2} .\end{split}\end{equation*}
By using stopping times as in the proof of Lemma \ref{lem1} we may assume that $$\E^\xi \exp\bigg[\ff \vv{T\e^{1+cT}} \int_0^T\bar l_t\d t \bigg]
<\infty$$ so that
$$\E^\xi \exp\bigg[\ff \vv{T\e^{1+cT}} \int_0^T\bar l_t\d t \bigg]\le \e^{2+ 2\vv l(0)}.$$ This completes the proof
by noting that
$$   \ff{1}{2\|\si\|^2\dd T^2\e^{2+2cT}} \int_0^T  \|W(X_t,Y_t)\|_\infty\d t \le \ff{r_0 \|W(\xi)\|_\infty}{2\|\si\|^2\dd T^2\e^{2+2cT}}
+\ff  \vv{T\e^{1+cT}} \int_0^T\bar l_t\d t .$$
\end{proof}

\beg{proof}[Proof of Corollary \ref{C1.3}] Let $v$ and $\aa$ be given in Corollary \ref{C1.2}.  By the semigroup property and the Jensen inequality, we will only consider  $T-r_0\in (0,1].$

(1)  By (\ref{QQ}) and the definitions of $\aa$ and $v$,  there exists a constant
$C>0$ such that
\beg{equation}\label{NN0}\beg{split} &|v'(s)h_2(0)+\aa'(s)|\le C1_{[0,T-r_0]}(s)|h(0)|\Big(\ff 1 {T-r_0}+ \ff{\|M\|}{(T-r_0)^{2(k+1)}}\Big),\ \ s\in [0,T],\\
&|\Theta(s)|\le C|h(0)|\Big(1+ \ff{\|M\|}{(T-r_0)^{2k+1}}\Big),\ \ \ s\in [0,T],\\
&\|\Theta_s\|_\infty \le C\Big(\|h\|_\infty+ \ff{\|M\|\cdot |h(0)|}{(T-r_0)^{2k+1}}\Big),\ \ \ s\in [0,T].\end{split}\end{equation}
Therefore, it follows from $(A3)$ and $(A4)$  that
\beg{equation}\label{NN}\beg{split} |N(s)|\le &C1_{[0,T-r_0]}(s)|h(0)|\Big(\ff 1 {T-r_0}+ \ff{\|M\|}{(T-r_0)^{2(k+1)}}\Big)\\
&+ C\Big(\|h\|_\infty+ \ff{\|M\|\cdot |h(0)|}{(T-r_0)^{2k+1}}\Big) \|W(X_s,Y_s)\|_\infty^l\end{split}\end{equation} holds for some constant $C>0$. Combining this with Theorem \ref{T1.1} we obtain

\beg{equation*}\beg{split} |\nn_h P_T f(\xi)|&\le C\ss{P_T f^2(\xi)}\bigg(\E^\xi\int_0^T |N(s)|^2\d s\bigg)^{1/2}\\
&\le C\ss{P_T f^2(\xi)}\bigg\{|h(0)|\Big(1+\ff{\|M\|}{(T-r_0)^{2k+1}}\Big)\\
&\quad + \Big(\|h\|_\infty + \ff{\|M\|\cdot |h(0)|}{(T-r_0)^{2k+1}}\Big)\bigg(\int_0^T \E^\xi  \|W(X_s,Y_s)\|_\infty^{2l}\d s\bigg)^{1/2}\bigg\},
\end{split}\end{equation*} This completes the proof of (1) since due to Lemma \ref{lem1} one has
$$\E^\xi  \|W(X_s,Y_s)\|_\infty^{2l}\le 3  \|W(\xi)\|_\infty^{2l} \e^{Cs},\ \ s\in [0,T]$$ for some constant $C>0.$

 (2)  By Theorem \ref{T1.1} and the Young inequality (cf. \cite[Lemma 2.4]{ATW09}), we have
 \beq\label{H1}\begin{split} |\nn_h P_T f|(\xi)&\le r\big\{P_Tf\log f- (P_Tf)\log P_T f\big\}(\xi)\\
 &\quad+ rP_Tf(\xi) \log\E^\xi\e^{\ff 1 r \int_0^T \<N(s),
 (\si^*)^{-1}\d B(s)\>},\ \ r>0.\end{split}\end{equation}
 Next, it follows from (\ref{NN}) that

 \beq\label{H2}\beg{split} & \bigg(\E^\xi \exp\bigg[\ff 1 r \int_0^T \<N(s), (\si^*)^{-1}\d B(s)\>\bigg]\bigg)^2\le
  \E^\xi \exp\bigg[\ff{2\|\si^{-1}\|^2}{ r^{2}} \int_0^T |N(s)|^2\d s\bigg]\\
 &\le \exp\bigg[ \ff{C_1|h(0)|^2}{r^2} \Big(\ff 1 {T-r_0} +\ff{\|M\|^2}{(T-r_0)^{4k+3}}\Big)\bigg]\\
 &\qquad\times \E^\xi \exp\bigg[ \ff{C_1}{r^2}\Big(\|h\|_\infty^2+\ff{\|M\|^2|h(0)|^2}{(T-r_0)^{4k+2}}\Big)
 \int_0^T \|W(X_s, Y_s)\|_\infty^{2l}\d s\bigg],\ \ T\in (r_0, 1+r_0]\end{split}\end{equation} holds
  for some constant $C_1\in (0,\infty).$  Since $2l\in [0,1)$ and $T\le 1+r_0$, there exists a constant $C_2\in (0,\infty)$ such that

$$\bb \|W(X_s,Y_s)\|_\infty^{2l}
   \le \ff{(\ff{\|h\|_\infty^2}{r^2}\land 1)\|W(X_s,Y_s)\|_\infty}{2\|\si\|^2\dd T^2\e^{2+2cT}}
   +C_2\bb^{\ff 1 {1-2l}}  \Big(\ff{\|h\|_\infty^2}{r^2}\land 1\Big)^{-\ff{2l}{1-2l}}, \, \bb>0.
$$ Taking
$$\bb=\ff{C_1}{r^2}\bigg(\|h\|_\infty^2+\ff{\|M\|^2|h(0)|^2}{(T-r_0)^{4k+2}}\bigg),$$
  and applying Corollary \ref{C3.1}, we arrive at
 \beg{equation*}\beg{split}& \E^\xi \exp\bigg[\beta\int_0^T \|W(X_s, Y_s)\|_\infty^{2l}\d s\bigg]
 \le \exp\bigg[C_2\bb^{\ff 1 {1-2l}}    \Big(\ff{\|h\|_\infty^2}{r^2}\land 1\Big)^{-\ff{2l}{1-2l}}\bigg] \\
&\qquad\qquad\qquad\times  \bigg(\E^\xi\exp\bigg[\ff 1 {2\|\si\|^2\dd T^2\e^{2+2cT}}\int_0^T \|W(X_s,Y_s)\|_\infty\d s\bigg]\bigg)^{\ff{\|h\|_\infty^2}{r^2}\land 1}\\
 &\le \exp\bigg[\ff{C_3}{r^2} \bigg\{\|h\|_\infty^2\|W(\xi)\|_\infty +\Big(\|h\|_\infty^2+\ff{\|M\|^2|h(0)|^2}{(T-r_0)^{4k+2}}\Big)^{\ff 1 {1-2l}} \Big(\ff{r^2}{\|h\|_\infty^2}\lor 1\Big)^{\ff{2l}{1-2l}}\bigg\}\bigg]  \end{split}\end{equation*} for some constant $C_3\in
(0,\infty)$ and all   $T\in (r_0,1+r_0]$. Therefore, the desired
entropy-gradient estimate follows by combining this with (\ref{H1})
and (\ref{H2}).

 (3) Let $C'>0$ be such that $r\ge C'\bigg(\|h\|_\infty+\ff{\|M\|\cdot|h(0)|}{(T-r_0)^{2k+1}}\bigg)$ implies
 $$\ff{C_1}{r^2}\Big(\|h\|_\infty^2+\ff{\|M\|^2|h(0)|^2}{(T-r_0)^{4k+2}}\Big)\le \ff 1 {2\|\si\|^2\dd T^2\e^{2+2cT}},$$ so that by Corollary \ref{C3.1}
 \beg{equation*}\beg{split} &\E^\xi \exp\bigg[\ff{C_1}{r^2}\Big(\|h\|_\infty^2+\ff{\|M\|^2|h(0)|^2}{(T-r_0)^{4k+2}}\Big)\int_0^T \|W(X_s, Y_s)\|_\infty^{2l}\d s\bigg] \\
 &\le\bigg(\E^\xi\exp\bigg[\ff 1 {2\|\si\|^2\dd T^2\e^{2+2cT}}\int_0^T \|W(X_s,Y_s)\|_\infty\d s\bigg]\bigg)^{\ff{2C_1\|\si\|^2\dd T^2\e^{2+2cT}}{r^2}\Big(\|h\|_\infty^2+\ff{\|M\|^2|h(0)|^2}{(T-r_0)^{4k+2}}\Big)}\\
 &\le \exp\bigg[\ff{C\|W(\xi)\|_\infty}{r^2}\Big(\|h\|_\infty^2+\ff{\|M\|^2|h(0)|^2}{(T-r_0)^{4k+2}}\Big)\bigg]\end{split}\end{equation*} holds for some constant $C>0$. Then proof is finished by combining this with (\ref{H1}) and (\ref{H2}).
\end{proof}

\beg{proof}[Proof of Theorem \ref{T1.5}] Again, we only prove for
$T\in (r_0,1+r_0].$ Applying (\ref{y1}) to $\vv=1$ and using $\log
f$ to replace $f$, we obtain \beq\label{W0} P_T \log f(\xi+h)=
\E\{R^1(T) \log f(X_T,Y_T)\}\le \log P_T f(\xi) +\E(R^1\log
R^1)(T).\end{equation} Next, taking $\vv=1$ in (\ref{2.6}) and
letting $n\uparrow\infty$, we arrive at \beq\label{W1} \E (R^1\log
R^1)(T)\le \ff 1 2 \E_{\Q_1}\int_0^T |\si^{-1}\Phi^1(r)|^2\d
r.\end{equation} By $(A3'), (A4')$, (\ref{NN0}) and the definition
of $\Phi^1$, we have \beg{equation*}\beg{split}
|\si^{-1}\Phi^1(s)|^2\le &C_1\bigg\{ \|W(X_s^1,Y_s^1)\|_\infty^{2l}+
U^2\Big(C_1 \|h\|_\infty
+\ff{C_1\|M\|\cdot|h(0)|}{(T-r_0)^{2k+1}}\Big)
 \bigg\}\|h\|_\infty^2 \\
 &+ C_1|h(0)|^2\bigg(\ff{1}{(T-r_0)^2}+\ff{\|M\|^2}{(T-r_0)^{4(k+1)}}\bigg)1_{[0,T-r_0]}(s)\end{split}\end{equation*} for some constant $C_1>0$. Then the proof is completed by combining this with (\ref{W0}), (\ref{W1}) and Lemma \ref{lem1} (note that  $(X^1(s),Y^1(s))$ under $\Q_1$ solves the same equation as $(X_s,Y_s)$  under $\P$).
 \end{proof}

\section{   Discrete Time Delay Case and Examples }

In this section we first present a simple example to illustrate our main results presented in Section 1, then relax assumption {\bf (A)} for the discrete time delay case in order to cover some highly non-linear examples.

\beg{exa}{\rm For $\alpha\in C([-r_0,0];\mathbb{R})$, consider
functional SDE on $\mathbb{R}^2$ \beg{equation} \beg{cases}
\d X(t)=-\{X(t)+Y(t)\}\d t\\
\d Y(t)=\d
B(t)+\Big\{-\vv Y^{3}(t)+Y(t-r_0)+\int_{-r_0}^0\alpha(\theta)X(t+\theta)\d\theta\Big\}\d
t
\end{cases}
\end{equation}
with initial data $\xi=(\xi_1,\xi_2)\in C([-r_0,0];\mathbb{R}^2)$, where $\vv\ge 0$ and $n\in \mathbb N$ are constants. For
$z=(x,y)\in\mathbb{R}^2$, let $W(x,y)=1+|x|^2+|y|^2$ and set
$Z(z)=-y^3$ and
$b(\xi)=\int_{-r_0}^0\alpha(\theta)\xi_1(\theta)\d\theta+\xi_2(-r_0)$.
By a straightforward computation one has for $x,y\in\mathbb{R}$
\beg{equation*} LW(x,y)=1-2 x(x+y)-2\vv y^{2n}\leq3W(x,y)
\end{equation*}
and for $\xi\in C([-r_0,0];\mathbb{R}^2)$ \beg{equation*}
\beg{split} \langle b(\xi),
\nn^{(2)}W(\xi(0))\rangle&\leq2\Big|\int_{-r_0}^0\alpha(\theta)\xi_1(\theta)\d\theta+\xi_2(-r_0)\Big||\xi_2(0)|\\
&\leq 2 \Big(1+\int_{-r_0}^0 \alpha (\theta)\d\theta\Big)\|\xi\|^2_\infty.
\end{split}
\end{equation*}
Then conditions (A1) and (A2) hold. Next, there exists a constant $c>0$ such that for any $z=(x,y)$ and
$z'=(x',y')\in\mathbb{R}^2$, $$
|Z(z)-Z(z')|=\vv|y^3-y'^3|\le c|y-y'|(|y'|^2+|y-y'|^2).$$
Finally,  for $\xi=(\xi_1,\xi_2),\xi'=(\xi_1',\xi_2')\in
C([-r_0,0];\mathbb{R}^2)$, \beg{equation*}\begin{split}
|b(\xi)-b(\xi')|
&\leq\sqrt{2}\Big(\int_{-r_0}^0|\alpha(\theta)|\d\theta\vee1\Big)\|\xi-\xi'\|_\infty.
\end{split}
\end{equation*}
So,  $(A3)$ holds for $l=1$ whenever $|y-y'|\leq1$ and $(A4)$ holds
for any $l\ge 0$. Moreover, $(A3')$ and $(A4')$ hold for $U(|z|)=|z|^2,
z\in\mathbb{R}^2$. Therefore,   Theorem \ref{T1.1}, Theorem
\ref{T1.5} and Corollary \ref{C1.3} hold. }
\end{exa}

To derive the entropy-gradient estimate and the Harnack inequality as in Corollary \ref{C1.4}, we need to weaken the assumption {\bf (A)}. To this end, we consider
a simpler setting where the delay is time discrete.
Consider
\beq\label{E20} \beg{cases} \d X(t)= \{AX(t)+MY(t)\}\d t,\\
\d Y(t)=Z(X(t),Y(t))+\tilde{b}(X(t-r_0),Y(t-r_0))\d t +\si \d B(t),
\end{cases}
\end{equation}
with initial data $\xi\in\mathscr{C}$, where
$Z,\tilde{b}:\mathbb{R}^{m+d}\rightarrow\mathbb{R}^d$. If we define
$b(\xi)=\tilde{b}(\xi(-r_0))$ for $\xi=(\xi_1,\xi_2)\in\mathscr{C}$,
then equation \eqref{E20} can be written as equation \eqref{E1}. For
$(x,y),(x',y')\in \R^{m+d}$, define the diffusion operator
associated with \eqref{E20} by
$$\L W(x,y;x',y')=   L W(x,y)  +\<\tilde{b}(x',y'), \nn^{(2)}W(x,y)\>.$$

\begin{thm}\label{T4.2}
Assume that there exist constants $\alpha,\beta,\gamma>0$ with
$\beta\geq\gamma$, functions $W\in C^2(\mathbb{R}^{m+d})$ with
$W\geq1$ and $U\in C(\mathbb{R}^{m+d};\mathbb{R}_+)$ such that for
$(x,y),(x',y')\in \R^{m+d}$ \beg{equation}\label{E21}
\L W(x,y;x',y')\leq\alpha \{W(x,y)+W(x',y')\}-\beta U(x,y)+\gamma
U(x',y').
\end{equation}
Assume further that there exists $\nu>0$ such that for
$z=(x,y),z'=(x',y')\in\mathbb{R}^{m+d}$ with $|z-z'|\leq1$
\beq\label{E25}|Z(z)-Z(z')|^2\vee|\tilde{b}(z)-\tilde{b}(z')|^2\leq
\nu|z-z'|^2W(z').
\end{equation}
Then for $\delta:=(\alpha r_0+1)\|W(\xi)\|_\infty+\gamma
r_0\|U(\xi)\|_\infty$ and $t\geq0$ \beq\label{E24} \mathbb{E}^\xi
W(X(t),Y(t))\leq\delta e^{2\alpha t},
\end{equation}
and
\beg{equation} \label{E23}\beg{split}|\nn_h P_Tf(\xi)|\le C &\ss{P_T f^2(\xi)} \bigg\{ |h(0)|\Big(1 + \ff{\|M\|}{(T-r_0)^{2k+1}\land 1}\Big)+r_0^{\frac{1}{2}}\|W(\xi)\|_\infty^{\frac{1}{2}}\|h\|_\infty\\
&+|h(0)|\ss{\delta(T\land (1+r_0))} \Big(1+
\ff{\|M\|}{(T-r_0)^{2k+1}}\Big)\bigg\}
\end{split}\end{equation}
 for all $T>r_0, \xi,h\in\C$ and $f\in
\B_b(\C)$, where $C>0$ is some constant. If moreover there exist
constants $K,\lambda_i\geq0,i=1,2,3,4,$ with
$\lambda_1\geq\lambda_2$ and $\lambda_3\geq\lambda_4$, functions
$\tilde{W}\in C^2(\mathbb{R}^{m+d})$ with $\tilde{W}\geq1$ and
$\tilde{U}\in C(\mathbb{R}^{m+d};\mathbb{R}_+)$  such that for
$(x,y),(x',y')\in \R^{m+d}$ \beq\label{E28} \frac{
\L \tilde{W}(x,y;x',y')}{\tilde{W}(x,y)}\leq K-
\lambda_1W(x,y)+\lambda_2W(x',y')-
\lambda_3\tilde{U}(x,y)+\lambda_4\tilde{U}(x',y'),
\end{equation}
then there exist constants $\delta_0,C>0$ such that for
$r\geq\delta_0/(T-r_0)^{2k+1}, \xi,h\in\C$ and positive $f\in
\B_b(\C)$
\begin{equation}\label{E26}
\begin{split}
&|\nn_h P_T f|(\xi)\le r\big\{P_Tf\log f- (P_Tf)\log P_T
f\big\}(\xi)\\
& +\frac{CP_Tf}{2r}\bigg\{ |h(0)|^2 \Big(\ff 1 {(T-r_0)\wedge1}
+\ff{\|M\|^2}{\{(T-r_0)\wedge1\}^{4k+3}}\Big)\\
& +\ff{(1+\|M\|^2)|h(0)|^2}{
\{(T-r_0)\wedge1\}^{4k+2}}\Big(\lambda_2r_0\|W(\xi)\|_\infty+\lambda_4r_0\|\tilde{U}(\xi)\|_\infty+KT+\log\tilde{
W}(\xi(0))\Big)\bigg\}.
\end{split}
\end{equation}
\end{thm}

\begin{proof}
By the It\^o formula one has for any $t\geq0$
\begin{equation*}
\begin{split}
\E^\xi W(X(t),Y(t)) &\leq
W(\xi(0))+\alpha\E^\xi\int_0^t\{W(X(s),Y(s))+W(X(s-r_0),Y(s-r_0))\}\d s\\
&\quad-\beta\E^\xi\int_0^tU(X(s),Y(s))ds+\gamma\E^\xi\int_0^tU(X(s-r_0),Y(s-r_0))ds\\
&\leq W(\xi(0))+\alpha\int_{-r_0}^0W(X(s),Y(s))ds+\gamma\int_{-r_0}^0U(X(s),Y(s))ds\\
&\quad+2\alpha\E^\xi\int_0^tW(X(s),Y(s))ds\\
 &\leq \delta +2\alpha\E^\xi\int_0^tW(X(s),Y(s)ds.
\end{split}
\end{equation*}
Then \eqref{E24} follows from the Gronwall inequality.

By Theorem \ref{T1.1}, for $T-r_0\in (0,1]$ and some $C>0$ we can
deduce that
 $$|\nn_h P_T f(\xi)|\le C\ss{P_T f^2(\xi)}\bigg(\E^\xi\int_0^T |N(s)|^2\d
 s\bigg)^{1/2},$$
where for $s\in [0,T]$
$$N(s):=  (\nn_{\Theta(s)}Z)(X(s), Y(s))
 + (\nn_{\Theta(s-r_0)}\tilde{b})(X(s-r_0), Y(s-r_0))-v'(s)h_2(0)-\aa'(s).$$
Recalling the first two inequalities in \eqref{NN0} and combining
\eqref{E25} yields that for some $C>0$ \beg{equation*}\beg{split}
|\nn_h P_T f(\xi)| &\le C\ss{P_T
f^2(\xi)}\bigg\{\bigg(\int_0^T|v'(s)h_2(0)+\aa'(s)|^2ds\bigg)^{1/2}\\
&\quad+\bigg(\E^\xi\int_0^T
|\Theta(s)|^2W(X(s),Y(s))\d s\bigg)^{1/2}\\
&\quad+\bigg(\E^\xi\int_0^T
|\Theta(s-r_0)|^2W(X(s-r_0),Y(s-r_0))\d s\bigg)^{1/2}\bigg\}\\
&\le C\ss{P_T f^2(\xi)}\bigg\{|h(0)|\Big(1+\ff{\|M\|}{(T-r_0)^{2k+1}}\Big)+r_0^{\frac{1}{2}}\|W(\xi)\|_\infty^{\frac{1}{2}}\|h\|_\infty\\
&\quad +|h(0)|\Big(1+ \ff{\|M\|}{(T-r_0)^{2k+1}}\Big)\bigg(\int_0^T
\E^\xi W(X(s),Y(s))\d s\bigg)^{1/2}\bigg\}.
\end{split}\end{equation*}
This, together with \eqref{E24}, leads to \eqref{E23}.

 Due to
\eqref{H1} and \eqref{H2} we can deduce that there exists $C>0$ such
that for arbitrary $r>0$ and $T-r_0\in(0,1]$
\beq\label{E30}\beg{split} & |\nn_h
P_T f|(\xi) \le r\big\{P_Tf\log f- (P_Tf)\log P_T f\big\}(\xi)\\
& +\frac{rP_Tf(\xi)}{2}\bigg\{ \ff{C|h(0)|^2}{r^2} \Big(\ff 1
{T-r_0}
+\ff{\|M\|^2}{(T-r_0)^{4k+3}}\Big)+\frac{C\|h\|_\infty^2\|W(\xi)\|_\infty r_0}{r^2}\\
& +\log\E^\xi \exp\bigg[
\ff{C(1+\|M\|^2)|h(0)|^2}{r^2(T-r_0)^{4k+2}}\int_0^T W(X(s), Y(s))\d
s\bigg]\bigg\}.
\end{split}\end{equation}
Moreover, since for $s\in[0,T]$ $$
\tilde{W}(X(s),Y(s))\exp\bigg(-\int_0^s\frac{
\L \tilde{W}(X(r),Y(r),X(r-r_0),Y(r-r_0))}{\tilde{W}(X(r),Y(r))}\d
r\bigg)
$$
is a local martingale by the It\^o formula, in addition to
$\tilde{W}\geq1$, we obtain from \eqref{E28} that
\beq\label{E29}\beg{split} &\E^\xi
\exp\bigg[(\lambda_1-\lambda_2)\int_0^T W(X(s), Y(s))\d
s-\lambda_2r_0\|W(\xi)\|_\infty\bigg]\\
&\leq\E^\xi \exp\bigg[\int_0^T \Big(\ll_1 W(X(s), Y(s))-\ll_2W(X(s-r_0), Y(s-r_0))\Big) \d
s\bigg]\\
&\leq\E^\xi \exp\bigg[
KT-\int_0^T\frac{\L \tilde{W}(X(s),Y(s);X(s-r_0), Y(s-r_0))}{\tilde{W}(X(s),Y(s))}\d s\\
&\qquad\qquad\quad-\lambda_3\int_0^T\tilde{U}(X(s),Y(s))\d s+\lambda_4\int_0^T\tilde{U}(X(s-r_0),Y(s-r_0))ds\bigg]\\
&\leq
\exp( \lambda_4r_0\|\tilde{U}(\xi)\|_\infty+
KT)\\
&\quad\times\E^\xi
\bigg[\tilde{W}(X(T),Y(T))\exp\bigg(-\int_0^T\frac{\L \tilde{W}(X(s),Y(s); X(s-r_0), Y(s-r_0))}{\tilde{W}(X(s),Y(s))}\d s\bigg)\bigg]\\
&\leq
\exp( \lambda_4r_0\|\tilde{U}(\xi)\|_\infty+
KT)\tilde{W}(\xi(0)).
\end{split}\end{equation}
Combining \eqref{E30} and \eqref{E29}, together with the H\"older
inequality, yields \eqref{E26}.
\end{proof}

The next example shows that Theorem \ref{T4.2} applies to the equation (\ref{E20}) with a highly non-linear drift.

\beg{exa}\label{Ex4.2}{\rm Consider delay SDE on $\mathbb{R}^2$
\beg{equation} \beg{cases}
\d X(t)=-\{X(t)+Y(t)\}\d t\\
\d Y(t)=\d
B(t)+\Big\{-Y^3(t)+\dfrac{1}{4}Y^3(t-r_0)+\dfrac{1}{2}X(t)-Y(t)\Big\}\d
t
\end{cases}
\end{equation}
with initial data $\xi\in C([-r_0,0];\mathbb{R}^2)$. In this example
for $z=(x,y),z'=(x',y')\in\mathbb{R}^2$ let
$Z(z)=\frac{1}{2}x-y-y^3$ and $b(z')=\frac{1}{4}y'^3$. For
$W(x,y)=1+x^2+y^4$ it is easy to see that
\begin{equation*}
\begin{split}
\L W(x,y;x',y')
&=-2x(x+y)+4y^3\Big(\frac{1}{2}x-y-y^3+\frac{1}{4}y'^3\Big)\\
&\leq-x^2+y^2-4y^4-4y^6+y^3y'^3+2y^3x\\
&\leq y^2-4y^4-\frac{5}{2}y^6+\frac{1}{2}y'^6.
\end{split}
\end{equation*}
Then \eqref{E21} holds for $\beta=\frac{5}{2},\gamma=\frac{1}{2}$
and $U(x,y)=y^6$. Moreover for $z=(x,y),z'=(x',y')\in\mathbb{R}^2$
there exists $c>0$ such that
\begin{equation*}
\begin{split}
|Z(z)-Z(z')|^2\lor |b(z)-b(z')|^2
&\leq c|z-z'|^2(|y-y'|^4+|y'|^4).
\end{split}
\end{equation*}
Thus condition \eqref{E25} holds, Therefore, by Theorem \ref{T4.2} we obtain
\eqref{E23}.

 To derive (\ref{E26}), we take
$w(x,y)=\frac{1}{4}(x^2+y^4)+\frac{1}{10}xy$  and set
$\tilde{W}(x,y)=\exp(w(x,y)-\inf w)$. Compute for
$(x,y,x',y')\in\mathbb{R}^4$ \beg{equation*}\begin{split}
\frac{\L\tilde{W}}{\tilde{W}}(x,y,x',y')&=\L\log
\tilde{W}(x,y)+\frac{1}{2}|\partial_y\log \tilde{W}|^2(x,y)\\
&\leq-\Big(\frac{1}{2}x+\frac{1}{10}y\Big)(x+y)+\Big(y^3+\frac{1}{10}x\Big)\Big(\frac{1}{2}x-y-y^3+\frac{1}{4}y'^3\Big)+\frac{3}{2}y^2\\
&\qquad+\frac{1}{2}\Big(y^3+\frac{1}{10}x\Big)^2\\
 &\leq0.5((0.35)^2/\epsilon+1.4)^2-(0.2325-\epsilon)x^2-0.5y^4-0.175y^6+0.1375y'^6,
\end{split}\end{equation*}
where $\epsilon>0$ is some constant such that $0.2325-\epsilon>0$.
Then condition \eqref{E28}  holds. Therefore, by Theorem \ref{T4.2} we obtain
\eqref{E26}, which implies the Harnack inequality as in Corollary \ref{C1.4} according to \cite[Proposition 4.1]{GW}. }
\end{exa}

\end{document}